\documentclass{amsart}
\usepackage{amsfonts}
\usepackage{amssymb}
\usepackage{amsthm}
\usepackage{latexsym,amsmath}
\usepackage{graphicx}
\usepackage{bm}
\usepackage{mathrsfs}
\usepackage[numbers]{natbib}
\usepackage{tikz}
\usepackage{caption}
\usepackage{enumerate}
\usepackage[utf8x]{inputenc}
\usepackage{underscore}
\usepackage{hyperref}
\usepackage{pgfplots}
\usepackage{float}
\usetikzlibrary{plotmarks}
\usetikzlibrary{arrows,positioning}
\pgfplotsset{ticks=none}

\newtheorem{thm}{Theorem}[section]
\newtheorem{lem}[thm]{Lemma}

\newtheorem{prop}[thm]{Proposition}
\newtheorem{defn}[thm]{Definition}

\DeclareMathOperator{\rank}{rank}

\DeclareMathOperator{\sgn}{sgn}

\DeclareMathOperator{\sym}{sym}

\DeclareMathOperator{\tri}{tri}

\title[Upper semicontinuity of the lamination hull]{Upper semicontinuity of the lamination hull}
\author[Terence L. J. Harris]{Terence L. J. Harris}
\address{School of Mathematics and Statistics, University of New South Wales, Sydney, NSW 2052, Australia.}
\address{Department of Mathematics, University of Illinois, Urbana, IL 61801, U.S.A.}
\email{terence2@illinois.edu}
\thanks{{\it Maths Subject Classification (2000):} 49J45 \and 52A30}
\thanks{This research was supported by the Australian Research Council's Discovery Projects funding scheme (project DP140100531).}
\date{}
\numberwithin{equation}{section}
\begin{document}
\begin{abstract}
Let $K \subseteq \mathbb{R}^{2 \times 2}$ be a compact set, let $K^{rc}$ be its rank-one convex hull, and let $L(K)$ be its lamination convex hull. It is shown that the mapping $K \mapsto \overline{L(K)}$ is not upper semicontinuous on the diagonal matrices in $\mathbb{R}^{2 \times 2}$, which was a problem left by Kolář. This is followed by an example of a 5-point set of $2 \times 2$ symmetric matrices with non-compact lamination hull. Finally, another 5-point set $K$ is constructed, which has $L(K)$ connected, compact and strictly smaller than $K^{rc}$. \end{abstract}
\maketitle
\section{Introduction}\label{s:1}
 \begin{sloppypar}
\label{intro} Let $\mathbb{R}^{m \times n}$ denote the space of $m \times n$ matrices with real entries. Two matrices \mbox{$X,Y\in \mathbb{R}^{m \times n}$} with $\rank(X-Y) = 1$ are called \textit{rank-one connected}.  A set $\mathcal{S} \subseteq \mathbb{R}^{m \times n}$ is \textit{lamination convex} if 
\[ \lambda X + (1-\lambda)Y \in \mathcal{S} \text{ for all } \lambda \in [0,1], \]
whenever $X,Y \in \mathcal{S}$ are rank-one connected. For a set $K \subseteq \mathbb{R}^{m \times n}$, the smallest lamination convex set containing $K$ is denoted by $L(K)$.  \end{sloppypar}

This work contains a counterexample to a question posed in \cite{kolar}, concerning the continuity of the mapping $K \mapsto \overline{L(K)}$ on $\mathbb{R}^{2 \times 2}$. The example is similar to Example 2.2 in \cite{aumann}. This is followed by a 5-point set $K$ of symmetric $2 \times 2$ matrices with non-compact $L(K)$, similar to Example 2.4 in \cite{kolar}. Then, another 5-point set $K$ is constructed which has $L(K)$ connected, compact and strictly smaller than $K^{rc}$. This is contrasted with Proposition 2.5 in \cite{sverak2}, which says that $K^{pc} =L(K)=K$ if $K$ is connected, compact and has no rank-one connections. Finally, a weaker version of this result is given for sets with rank-one connections. 

\section{Main results}\label{s:2}
Define the Hausdorff distance between two compact sets $K_1,K_2$ in $\mathbb{R}^{m \times n}$ by 
\[ \rho(K_1,K_2) = \inf\{ \epsilon \geq 0 : K_1 \subseteq U_{\epsilon}(K_2) \text{ and } K_2 \subseteq U_{\epsilon}(K_1) \}, \]
where $U_{\epsilon}(K)$ is the open $\epsilon$-neighbourhood of $K$, corresponding to the Euclidean distance. Let $\mathcal{K}$ be the set of compact subsets of $\mathbb{R}^{m \times n}$. A function $f: \mathcal{K} \to \mathcal{K}$ is upper semicontinuous if for every $\epsilon>0$ and for every $K_0 \in \mathcal{K}$, there exists a $\delta>0$ such that $f(K) \subseteq U_{\epsilon}(f(K_0))$ whenever $\rho(K,K_0) <\delta$. It is known that the function $K \mapsto K^{rc}$ is upper semicontinuous on the compact subsets of $\mathbb{R}^{m \times n}$ (see for example the proof of Theorem 1 in \cite{szekelyhidi}, Example 4.18 in \cite{kirchheim}, or Theorem 3.2 in \cite{zhang}.) The following example (pictured in Figure \ref{firstfigure}) shows that this fails on diagonal matrices in $\mathbb{R}^{2 \times 2}$, for the lamination convex hull. 
\begin{thm} \label{uppersemthm} There exists a compact set $K_0$ of diagonal matrices in $\mathbb{R}^{2 \times 2}$ such that the mapping $K \mapsto \overline{L(K)}$ is not upper semicontinuous at $K_0$.  \end{thm}
\begin{proof} Identify the space of $2 \times 2$ diagonal matrices with $\mathbb{R}^2$ in the natural way. Let
\[ K_0 = \left\{(1,0) \right \}\cup  \bigcup_{n=0}^{\infty} \left\{ \left( 1- \frac{3}{2^{n+1}}, \frac{1}{2^{n+1}} \right), \quad \left(1-\frac{1}{2^n}, \frac{3}{2^{n+1}} \right)  \right\}. \]
The set $K_0$ is compact and has no rank-one connections, thus $L(K_0)=K_0$. For each integer $n \geq -1$ let 
\[ P_n =\left(1-\frac{1}{2^{n+1}}, \frac{1}{2^{n+1}} \right). \] Given $\delta > 0$, choose a positive integer $N$ large enough to ensure that $\frac{1}{2^{N}} < \delta$, and let $K = K_0 \cup \{ P_N\}$, so that $\rho(K,K_0) < \delta$. 
Then 
\[ \left( 1- \frac{1}{2^N}, \frac{1}{2^{N+1}} \right) = \frac{1}{2} \left(1-\frac{3}{2^{N+1}}, \frac{1}{2^{N+1}} \right) + \frac{1}{2} P_N \in L(K), \]
and hence
\[ P_{N-1} = \frac{1}{2}\left( 1- \frac{1}{2^N}, \frac{1}{2^{N+1}} \right) + \frac{1}{2}  \left(1-\frac{1}{2^N}, \frac{3}{2^{N+1}} \right) \in  L(K). \]
It follows by induction that $\left( 0, 1 \right) = P_{-1} \in L(K)$. Since $\rho(P_{-1},L(K_0)) \geq \frac{1}{2}$, this shows that the function $K \mapsto \overline{L(K)}$ is not upper semicontinuous at $K_0$. 
\end{proof}
\begin{figure}[htbp]
\begin{tikzpicture}[scale = 4] 
  \fill (1-3*.5,.5) circle[radius=.3pt];
	\fill (1-3*.25,.25) circle[radius=.3pt];
	\fill (1-3*.125,.125) circle[radius=.3pt];
	\fill (1-3*.0625,.0625) circle[radius=.3pt];
	\fill (1-3*.03125,.03125) circle[radius=.3pt];
	\fill (1-3*.015625,.015625) circle[radius=.3pt];
	\fill (1-3*.0078125,.0078125) circle[radius=.3pt];
	
	\fill (1-1,3*.5) circle[radius=.3pt];
	\fill (1-.5,3*.25) circle[radius=.3pt];
	\fill (1-.25,3*.125) circle[radius=.3pt];
	\fill (1-.125,3*.0625) circle[radius=.3pt];
	\fill (1-.0625,3*.03125) circle[radius=.3pt];
	\fill (1-.03125,3*.015625) circle[radius=.3pt];
	\fill (1-.015625,3*.0078125) circle[radius=.3pt];
	
	\fill (1,0) circle[radius=.3pt];
	
	\draw [dotted] (1-3*.5,.5) -- (1-.5,.5);
	\draw [dotted] (1-3*.5,.5) -- (1-.5,.5);
	\draw [dotted] (1-3*.25,.25) -- (1-.25,.25);
	\draw [dotted] (1-3*.125,.125) -- (1-.125,.125);
	\draw [dotted] (1-3*.0625,.0625) -- (1-.0625,.0625);
	\draw [dotted] (1-3*.03125,.03125) -- (1-.03125,.03125);
	
	\draw [dotted] (1-1,.5) -- (1-1,3*.5);
	\draw [dotted] (1-.5,.25) -- (1-.5,3*.25);
	\draw [dotted] (1-.25,.125) -- (1-.25,3*.125);
	\draw [dotted] (1-.125,.0625) -- (1-.125,3*.0625);
	\draw [dotted] (1-.0625,.03125) -- (1-.0625,3*.03125);
	\draw [dotted] (1-.03125,.015625) -- (1-.03125,3*.015625);

	
\end{tikzpicture}
\caption{The set $K_0$ from Theorem \ref{uppersemthm}. The dotted lines are rank-one lines in $L(K)$, where $K$ is a small perturbation of $K_0$. }
\label{firstfigure}
\end{figure}
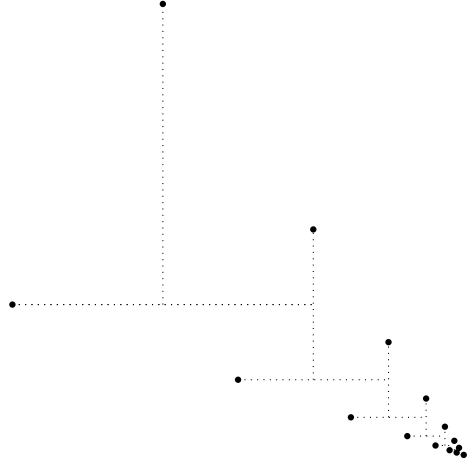

The next result gives two examples of 5-point subsets of $\mathbb{R}^{2 \times 2}$, each with a non-compact lamination hull. The upper-triangular example is pictured in Figure \ref{uppertriag}. It consists of 4 points in the diagonal plane arranged in a $T_4$ configuration, together with a point whose projection onto the diagonal plane is a corner of the inner rectangle of the $T_4$ configuration. 

Throughout, the upper triangular matrix $\begin{pmatrix} x & z \\ 0 & y \end{pmatrix}$ will be identified with the point $(x,y,z) \in \mathbb{R}^3$. The symmetric example uses essentially the same idea as in Figure \ref{uppertriag}, so the matrix $\begin{pmatrix} x & z \\ z & y \end{pmatrix}$ will also be denoted by the point $(x,y,z) \in \mathbb{R}^3$. Since the cases are treated separately, the notations do not conflict. The symmetric notation also differs from the usual identification, used for example in \cite{kolar}. The space of $2 \times 2$ upper triangular matrices is denoted by $\mathbb{R}^{2 \times 2}_{\tri}$, and the space of  $2 \times 2$ symmetric matrices by $\mathbb{R}^{2 \times 2}_{\sym}$. Up to linear isomorphisms preserving rank-one directions, these are the only two 3-dimensional subspaces of $\mathbb{R}^{2 \times 2}$ (see \cite[Corollary 6]{conti} or \cite[Lemma 3.1]{zimmerman})

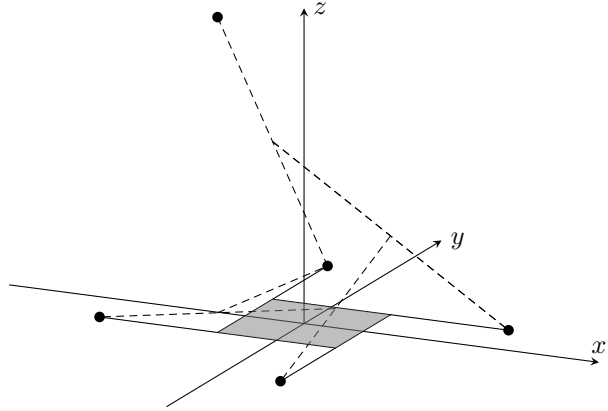
\begin{figure}[htbp]
\begin{tikzpicture}
\begin{axis}[
  axis lines=center,
  width=13cm,height=9cm,
  xmin=-5,xmax=5,ymin=-5,ymax=5,zmin=0,zmax=1,
]
\node [above] at (axis cs:5,0,0) {$x$};
\node [right] at (axis cs:0,5,0) {$y$};
\node [right] at (axis cs:0,0,1) {$z$};


\addplot3 [no marks] coordinates {(-1,1,0) (1,1,0) (1,-1,0) (-1,-1,0) (-1,1,0)};
\addplot3 [no marks] coordinates {(-1,3,0) (-1,1,0)};
\addplot3 [no marks] coordinates {(1,1,0) (3,1,0)};
\addplot3 [no marks] coordinates {(1,-1,0) (1,-3,0)};
\addplot3 [no marks] coordinates {(-1,-1,0) (-3,-1,0)};

\addplot3 [no marks,densely dashed] coordinates {(-1,-1,1) (-1,3,0)};
\addplot3 [no marks,densely dashed] coordinates {(-1,1,0.5) (3,1,0)};
\addplot3 [no marks,densely dashed] coordinates {(-1,1,0.5) (3,1,0)};
\addplot3 [no marks,densely dashed] coordinates {(1,1,0.25) (1,-3,0)};
\addplot3 [no marks,densely dashed] coordinates {(1,-1,0.125) (-3,-1,0)};
\addplot3 [no marks,densely dashed] coordinates {(-1,-1,0.0625) (-1,3,0)};

\fill [fill opacity=0.5,fill=gray] (axis cs:-1,1,0) -- (axis cs:1,1,0)--(axis cs:1,-1,0)--(axis cs:-1,-1,0);
\fill (axis cs:-1,3,0) circle[radius=2pt];
\fill (axis cs:3,1,0) circle[radius=2pt];
\fill (axis cs:1,-3,0) circle[radius=2pt];
\fill (axis cs:-3,-1,0) circle[radius=2pt];
\fill (axis cs:-1,-1,1) circle[radius=2pt];
\end{axis}
\end{tikzpicture}
\caption{A 5-point set $K \subseteq \mathbb{R}^{2 \times 2}_{\tri}$ together with 5 rank-one lines in $L(K)$. The dashed lines indicate rank-one lines in $L(K)$, which spiral toward the diagonal plane and make $L(K)$ non-compact. }
\label{uppertriag}
\end{figure}
\begin{thm} \leavevmode 
\begin{enumerate}[(i)]
\item There exists a 5-point set $K \subseteq \mathbb{R}^{2 \times 2}_{\tri}$ such that $L(K)$ is not compact. 
\item There exists a 5-point set $K \subseteq \mathbb{R}^{2 \times 2}_{\sym}$ such that $L(K)$ is not compact. 
\end{enumerate}
\end{thm}
\begin{sloppypar}
\begin{proof} For part (i) let $x_1 < x_2$, $y_2< y_1$, $z_0 > 0$ and $\alpha_0, \alpha_1,\alpha_2, \alpha_3 >0$.  Let 
\[ P_0 = (x_1,y_2,0), \quad P_1=(x_1,y_1,0), \quad P_2=(x_2,y_1,0), \quad P_3=(x_2,y_2,0), \] 
and set
\begin{align*} A_0 &= (x_1,y_1+ \alpha_0,0),  \quad A_1 = (x_2+ \alpha_1,y_1,0),  \\
A_2 &= (x_2,y_2-\alpha_2,0), \quad A_3 = (x_1-\alpha_3,y_2,0). \end{align*}
For $i\in \{0,1,2,3\}$ let $A_4=A_0$ and
\[ \lambda_i = \frac{\det(A_i-A_{i+1})}{\det(A_i-A_{i+1})-\det(P_i-A_{i+1})} \in (0,1),\]
let $X_0 = P_0 + (0,0,z_0)$ and $K = \{A_0,A_1,A_2,A_3,X_0\}$. For $i \geq 0$ let 
\[ X_{i+1} = (1-\lambda_{i \bmod 4}) A_{i \bmod 4} + \lambda_{ i \bmod 4} X_i, \]
so that for $i \geq 0$ and $k \in \{0,1,2,3\}$, induction gives 
\[ X_{4i+k} = P_k +(\lambda_0\lambda_1 \lambda_2 \lambda_3)^i \left(\prod_{j=0}^{k-1} \lambda_j \right) (0,0,z_0), \quad  \det(X_i-A_{i \bmod 4}) = 0,\]
which implies that $X_i \in L(K)$ for every $i \geq 0$. Hence $P_0 \in \overline{L(K)}$, and it remains to show that $P_0 \notin L(K)$. This follows from the fact that
\[ \{ (x,y,z) \in \mathbb{R}^{2 \times 2}_{\tri} : z >0 \} \cup \{A_0, A_1,A_2,A_3\} \]
is a lamination convex set containing $K$, which does not contain $P_0$.

For part (ii), let all the scalars and diagonal points be the same as in part (i). Using the symmetric notation let $Y_0 = P_0 + \left(\xi_1,\xi_2,\xi_3 \right)$ where $\xi_3 >0$ and 
\[ \xi_1 = \frac{1}{2} \left( -\alpha_3 + \sqrt{\alpha_3^2 - \frac{4\alpha_3 \xi_3^2}{y_1+\alpha_0-y_2} } \right), \quad \xi_2 =  \frac{-\xi_1(y_1+\alpha_0-y_2)}{\alpha_3}. \]
so that $\det(Y_0-A_0)=\det(Y_0-A_3)=0$, and $Y_0 \to P_0$ as $\xi_3 \to 0$. 
The fact that $\det(P_0-A_1)>0>\det(A_0-A_1)$ means that
\[ \det(Y_0 - A_{1}) > 0 >  \det(A_0 - A_{1}), \]
whenever $\xi_3 \in (0, \epsilon_1)$, for some $\epsilon_1 >0$. Set $B_0=Y_0$. For $i \in \{0,1,2,3\}$ and $B_i$ with
\[\det(B_i - A_{i+1}) \neq 0 \text{ and } \sgn \det(B_i - A_{i+1}) \neq \sgn \det(A_i - A_{i+1}), \]
let 
\[ B_{i+1} = (1-t_i)A_i + t_iB_i, \text{ where } t_i = \frac{\det(A_i-A_{i+1})}{\det(A_i-A_{i+1})-\det(B_i-A_{i+1})} \in (0,1), \]
so that $\det(B_{i+1}- A_{i+1}) =0$. By induction $t_i \to \lambda_i$ as $\xi_3 \to 0$ for $i \in \{0,1,2,3\}$, $B_i \to P_{i \bmod 4}$ as $\xi_3 \to 0$ for each $i \in \{0,1,2,3,4\}$, and $B_1, B_2, B_3, B_4$ all exist if $\xi_3$ is sufficiently small. Hence there exists $\epsilon_2>0$ such that $(t_0t_1t_2t_3) < \frac{1}{2} \left( 1+\lambda_0\lambda_1\lambda_2\lambda_3\right)$ and $B_1, B_2, B_3, B_4$ all exist whenever $\xi_3 \in (0,\epsilon_2)$. Put \mbox{$(\eta_1, \eta_2, \eta_3) = B_4 - P_0$}. Then since $\det(B_4-A_0)=\det(B_4-A_3)=0$,
\begin{equation} \label{bfour} \eta_1 = \frac{1}{2} \left( -\alpha_3 \pm \sqrt{\alpha_3^2 - \frac{4\alpha_3 \eta_3^2}{y_1+\alpha_0-y_2} } \right), \quad \eta_2 =  \frac{-\eta_1(y_1+\alpha_0-y_2)}{\alpha_3}. \end{equation}
But since $B_4 \to P_0$ as $\xi_3 \to 0$, there exists $\epsilon_3>0$ such that the sign in \eqref{bfour} is positive whenever $\xi_3 \in (0,\epsilon_3)$. Let $\epsilon = \min\{\epsilon_1, \epsilon_2, \epsilon_3\}$. If $\xi_3 \in (0, \epsilon)$. then 
\begin{equation} \label{iteratestart} \eta_3 = (t_0t_1t_2t_3)\xi_3 < \frac{1}{2}(1+ \lambda_0\lambda_1 \lambda_2\lambda_3) \xi_3. \end{equation}
Therefore let $K= \{A_0,A_1,A_2,A_3,Y_0\}$, and set $Y_1 = B_4$. Then $Y_1 \in L(K)$ by the preceding working. By \eqref{iteratestart}, iterating this process gives a sequence $Y_n \in L(K)$ with $Y_n \to P_0 \in \overline{L(K)}$. Again the point $P_0$ is not in $L(K)$ since
\[ \{ (x,y,z) \in \mathbb{R}^{2 \times 2}_{\sym} : z >0 \} \cup \{A_0, A_1,A_2,A_3\} \]
is a lamination convex set separating $P_0$ from $K$.  Hence $L(K)$ is not compact.   
\end{proof} \end{sloppypar}

\begin{sloppypar}
 A function \mbox{$f: \mathbb{R}^{m \times n} \to \mathbb{R}$} is called \textit{rank-one convex} if 
\[ f(\lambda X + (1-\lambda)Y) \leq \lambda f(X)+(1-\lambda)f(Y) \text{ for all } \lambda \in [0,1], \]
whenever $\rank(X-Y) \leq 1$.  The rank-one convex hull of a compact set \mbox{$K \subseteq \mathbb{R}^{m \times n}$} is defined by
\[ K^{rc} = \{ X \in \mathbb{R}^{m \times n} :  f(X) \leq 0 \quad \forall \text{ rank-one convex } f \text{ with } f|_K \leq 0 \}. \]
The polyconvex hull is defined similarly via polyconvex functions;  a function \mbox{$f: \mathbb{R}^{2 \times 2} \to \mathbb{R}$} is polyconvex if there exists a convex function \mbox{$g: \mathbb{R}^{2 \times 2} \times \mathbb{R} \to \mathbb{R}$} such that \mbox{$f(X) = g(X, \det X)$} for all $X \in \mathbb{R}^{2 \times 2}$. For compact $K$, the following characterisation of $K^{pc}$ will be used (see Theorem 1.9 in \cite{kirchheim}):
\begin{equation} \label{barycentre} K^{pc} = \{ \overline{\mu} : \mu \in \mathscr{M}_{pc}(K) \}, \end{equation}
where $\mathscr{M}_{pc}(K)$ is the class of probability measures supported in $K$ which satisfy Jensen's inequality for all polyconvex $f$;
\[ f(\overline{\mu}) \leq \int_{\mathbb{R}^{2 \times 2} } f(X) \:d\mu(X) \quad \text{ where } \quad \overline{\mu} = \int_{\mathbb{R}^{2 \times 2} }X \: d\mu(X). \]  \end{sloppypar}
\begin{defn} \label{tfourdefn} An ordered set $\{X_i\}_{i=1}^4 \subseteq \mathbb{R}^{m \times n}$ without rank-one connections is called a $T_4$ configuration if there exist matrices $P,C_1,C_2,C_3,C_4 \in \mathbb{R}^{m \times n}$ and real numbers $\mu_1,$ $\mu_2,$ $\mu_3,$ $\mu_4 >1$ satisfying
\[\rank C_i = 1 \text{ for } 1 \leq i \leq 4, \quad \sum_{i=1}^4 C_i = 0, \]
and
\begin{align}\notag  X_1 &= P+ \mu_1 C_1 \\
\notag X_2 &= P+ C_1+ \mu_2 C_2 \\
\notag X_3 &= P+ C_1+C_2+ \mu_3 C_3 \\
\label{t4defn} X_4 &= P+ C_1+C_2+C_3+ \mu_4 C_4. \end{align}
An unordered set $\{X_i\}_{i=1}^4$ is a $T_4$ configuration if it has at least one ordering which is a $T_4$ configuration. \end{defn}

The following result is a slight generalisation of Theorem 1 in \cite{szekelyhidi} (see also Corollary 3 in \cite{faraco}). The proof is similar to the one in \cite{szekelyhidi}, with minor technical changes.
\begin{thm} \label{mainthm} If $K \subseteq \mathbb{R}^{2 \times 2}$ is compact, and does not have a $T_4$ configuration $\{X_i\}_{i=1}^4$ with at least two $X_i,X_j$ in distinct connected components of $L(K)$, then
\[ K^{rc} = \bigcup_{i} (U_i \cap K)^{rc} \quad \text{ and } \quad  K^{qc} = \bigcup_{i} (U_i \cap K)^{qc}, \]
where the $U_i$ are the connected components of $L(K)$. \end{thm}
On diagonal matrices the conclusion reduces to $K^{rc} = L(K)$.  The following proposition shows that this fails in the full space $\mathbb{R}^{2 \times 2}$. \begin{sloppypar}
\begin{prop} \label{counter1}
There exists a 5-point set $K \subseteq \mathbb{R}^{2 \times 2}$ with $L(K)$ connected, compact and strictly smaller than $K^{rc}$.
\end{prop}\end{sloppypar}
 \begin{sloppypar}
\begin{proof} Fix $\epsilon \in (0,1)$, let
\[ X_1 = \begin{pmatrix} 1 & 0 \\ 0 & 0\end{pmatrix}, \quad  X_2 = \begin{pmatrix} 0 & 0 \\ 0 & 1\end{pmatrix}, \quad X_3 = \begin{pmatrix} -\epsilon & -1 \\ -\epsilon^2 & -\epsilon \end{pmatrix},\quad  X_4 = \begin{pmatrix} -\epsilon & \epsilon^2 \\ 1 & -\epsilon\end{pmatrix}, \]
and let
\begin{equation} \label{mudefns} \mu_1 =\frac{1+2\epsilon}{\epsilon(1-\epsilon^2)}, \quad \mu_2 = 1+\epsilon^2 \mu_1, \quad \mu_3 = 1+ \left(\frac{1+\epsilon^2}{\epsilon} \right)\mu_2, \quad \mu_4 = 1+ \epsilon^2 \mu_3, \end{equation}
so that 
\begin{equation} \label{mudefns2} \mu_1 = 1+ \frac{\mu_4}{\epsilon(1+\epsilon^2)}. \end{equation}
Set
\begin{align*} P_1 &= \frac{1}{\epsilon(\mu_1-1)} \begin{pmatrix} -\epsilon & 0 \\ 1 & 0 \end{pmatrix}, \quad P_2 = \frac{1}{\mu_1\epsilon} \begin{pmatrix} 0 & 0 \\ 1 & 0 \end{pmatrix}, \\
P_3 &=\frac{1}{\mu_2} \begin{pmatrix} 0 & 0 \\ \epsilon & 1 \end{pmatrix}, \quad P_4 = \frac{1}{\mu_3 \epsilon}\begin{pmatrix} -\epsilon^2 & -\epsilon  \\ \epsilon & 1 \end{pmatrix}, \end{align*}
and let $C_i = P_{i+1}-P_i$, where $P_5 := P_1$. Then clearly $\rank C_i = 1$ for all $i$, whilst \eqref{mudefns} and \eqref{mudefns2} imply that this is a solution of \eqref{t4defn}. 
Let
\[ K = \{0, X_1, X_2,X_3,X_4\}, \quad \text{ so that } \quad L(K)= \bigcup_{i=1}^4 \left[0, X_i \right]. \]
To prove the second formula for $L(K)$, it suffices to show that the set \mbox{$\mathcal{S} = \bigcup_{i=1}^4 \left[0, X_i \right]$} is lamination convex. For $i \neq j$, the fact that $\det X_i = \det X_j = 0$ and \mbox{$\det (X_i-X_j) \neq 0$} implies that $\det(X_i - tX_j) \neq 0$ whenever $t \in (0,1]$, since the determinant is linear along rank-one lines. It follows similarly that \mbox{$\det(sX_i - tX_j) \neq 0$} for $s,t \in (0,1]$, and so the only rank-one connected pairs in $\mathcal{S}$ are 0 and $tX_i$ for any $i$. Hence $\mathcal{S}$ is lamination convex. By Lemma 2 in \cite{szekelyhidi}, the point $P_1$ is in $K^{rc} \setminus L(K)$, so this proves the proposition. \end{proof} \end{sloppypar}

The preceding example contrasts with Lemma 3 in \cite{sverak2}, which states (in a weakened form) that $K^{pc} =K$ if $K$ is a connected compact subset of $\mathbb{R}^{2 \times 2}$ without rank-one connections. The example shows that the assumption that $K$ has no rank-one connections cannot be weakened to $L(K)=K$. The reason is that $\det(X-Y)$ cannot change sign on connected subsets of $\mathbb{R}^{2 \times 2}$ without rank-one connections, whilst it can on lamination convex sets. If the assumption that $\det(X-Y)$ does not change sign is added, $K^{pc}$ is equal to the lamination hull of order 2: given a set $K \subseteq \mathbb{R}^{m \times n}$, let $L^{(0)}(K)=K$ and define $L^{(k)}(K)$ inductively by 
\[ L^{(k+1)}(K) = \bigcup_{ \substack{X, Y \in L^{(k)}(K) \\ \rank(X-Y) \leq 1} } [X,Y]. \]
\begin{sloppypar}
\begin{prop} If $K \subseteq \mathbb{R}^{2 \times 2}$ is a compact set such that \mbox{$\det(X-Y) \geq 0$} for every $X,Y \in K$, then $K^{pc} = L^{(2)}(K)$. \end{prop}
\begin{proof} If $\mu$ is a probability measure supported in $K$ with $\det \overline{\mu} = \int_{\mathbb{R}^{2 \times 2}} \det X \: d \mu$, then as in \cite{sverak2},
\[ \int_{\mathbb{R}^{2 \times 2}} \int_{\mathbb{R}^{2 \times 2}} \det(X-Y) \: d\mu(X) \: d\mu(Y) = 0, \]
and therefore $\det(X-Y) =0$ whenever $X$ and $Y$ are in the support of $\mu$. This implies (see the following Lemma \ref{linalg}) that the support of $\mu$ is contained in a 2-dimensional affine plane $P$ consisting only of rank-one directions. Therefore $\overline{\mu} \in (K \cap P)^{co}$, and so Carathéodory's Theorem gives 3 points $X_i \in K \cap P$ such that $\overline{\mu}$ is a convex combination $ \overline{\mu}=\lambda_1 X_1 + \lambda_2 X_2 + \lambda_3 X_3$, and without loss of generality $\lambda_1 \neq 0$. Then $\frac{ \lambda_1}{\lambda_1+\lambda_2} \cdot X_1 + \frac{ \lambda_2}{\lambda_1+\lambda_2} \cdot X_2 \in P \cap L^{(1)}(K)$ since $P$ is a plane consisting of rank-one directions, and similarly
\[ \overline{\mu} = (\lambda_1+\lambda_2)\left( \frac{ \lambda_1}{\lambda_1+\lambda_2} \cdot X_1 + \frac{ \lambda_2}{\lambda_1+\lambda_2} \cdot X_2 \right) + \lambda_3 X_3 \in L^{(2)}(K). \]
It follows from \eqref{barycentre} that \mbox{$K^{pc} =L^{(2)}(K)$}. 
\end{proof}
\end{sloppypar}
\begin{lem} \label{linalg} Let $X_0, Y_0 \in \mathbb{R}^{m \times n}$ satisfy $\rank(X_0-Y_0) = 1$, and let
\[ \mathcal{S} = \{ X \in \mathbb{R}^{m \times n} : \rank(X-X_0) \leq 1 \text{ and } \rank(X-Y_0) \leq 1 \}. \] 
Then:
\begin{enumerate}[(i)]
\item $\mathcal{S} = P_1 \cup P_2$, where $P_1$ is an $m$-dimensional affine plane and $P_2$ is an $n$-dimensional affine plane, and for each fixed $i$, $\rank(X-Y) \leq 1 \text{ for } X,Y \in P_i$. 
\item The planes $P_1$ and $P_2$ satisfy 
\[ \rank(X-Y) > 1 \quad \text{ for } \quad X \in P_1 \setminus P_2 \quad \text{ and } \quad Y \in P_2 \setminus P_1. \] \end{enumerate}
\end{lem}
\begin{proof} By translation invariance it may be assumed that $Y_0=0$, so that \mbox{$\rank X_0=1$} and $X_0 = v_0w_0^T$ for some nonzero $v_0 \in \mathbb{R}^m$, $w_0 \in \mathbb{R}^n$. Let
\[ P_1 = \{ xw_0^T: x \in \mathbb{R}^m \}, \quad P_2 = \{ v_0y^T: y \in \mathbb{R}^n \}. \]  If $X \in \mathcal{S}$ then $X= vw^T$ for some $v \in \mathbb{R}^m$ and $w \in \mathbb{R}^n$, and
\begin{equation} \label{planeeq} X-X_0= vw^T - v_0w_0^T = ab^T, \end{equation}
for some $a \in \mathbb{R}^m$ and $b \in \mathbb{R}^n$. Suppose for a contradiction that $X \notin P_1 \cup P_2$. Then since $X \notin P_1$ there exists a vector $w_0^{\perp}$ such that $\langle w_0, w_0^{\perp} \rangle =0$ and $\langle w, w_0^{\perp} \rangle \neq 0$. Right multiplying both sides of \eqref{planeeq} with $w_0^{\perp}$ gives
\[ v = \frac{ \langle b, w_0^{ \perp} \rangle a }{\langle w, w_0^{\perp} \rangle }, \quad \text{ and similarly } \quad w = \frac{ \langle a, v_0^{ \perp} \rangle b }{\langle v, v_0^{\perp} \rangle }. \]
Let $\lambda = \frac{  \langle a, v_0^{ \perp} \rangle  \langle b, w_0^{ \perp} \rangle }{ \langle v, v_0^{\perp} \rangle \langle w, w_0^{\perp} \rangle }$. Then $\lambda \neq 1$ by \eqref{planeeq} since $v_0w_0^T \neq 0$, and therefore
\[ X=vw^T = \left(\frac{ \lambda }{\lambda -1 } \right) v_0w_0^T \in P_1 \cap P_2, \]
which is a contradiction. This proves part (i).

For part (ii), let $X= xw_0^T \in P_1 \setminus P_2$, let $Y = v_0 y^T \in P_2 \setminus P_1$ and suppose for a contradiction that $\rank(X-Y) = 1$. Then by part (i), $Y = xz^T$ for some nonzero $z \in \mathbb{R}^n$, and therefore $x = \frac{ v_0 \langle y,z \rangle }{\|z\|^2 }$, which contradicts the fact that $X \notin P_2$. 
\end{proof}

\section*{Acknowledgements}
I would like to thank Michael Cowling and Alessandro Ottazzi for advice on this topic, and for comments on the draft.


\end{document}